\documentclass[11pt]{amsart}
\usepackage{verbatim, amscd,epsfig,amsmath,amsthm,amstext,amssymb,
hyperref
}
 \usepackage{bbm}%,theorem}

\theoremstyle{plain}
\newtheorem*{theorem*}{Theorem}
\newtheorem{theorem}{Theorem}[section]
\newtheorem{lemma}[theorem]{Lemma}

\theoremstyle{definition}
\newtheorem{definition}[theorem]{Definition}

\theoremstyle{remark}
\newtheorem{rem}[theorem]{Remark}
\newtheorem{krem}{Note/Todo}
\newtheorem{remark}[theorem]{Remark}

\newcommand{\assign}{:=}

\numberwithin{equation}{section}  
\renewcommand{\Re}{{\rm Re}\,}
\renewcommand{\Im}{{\rm Im}\,}

\newcommand{\ejbh}{e^{\frac{j\beta}2}}
\newcommand{\emjbh}{e^{-\frac{j\beta}2}}

\newcommand{\Cj}{\mathfrak{C}}

\newcommand{\R}{\mathbb{R}}
\newcommand{\C}{\mathbb{ C}}
\newcommand{\Z}{\mathbb{ Z}}

\renewcommand{\H}{\mathbb{ H}}

\newcommand{\N}{\mathbb{ N}}

\newcommand{\trivial}[1]{\underline{\H}^{#1}}

\newcommand{\invers}{^{-1}}
\DeclareMathOperator{\II}{II}

\DeclareMathOperator{\Span}{Span}

\newcommand{\Hh}{{\mathcal H}}

\begin{document}
\title{Families of conformal tori of revolution in the 3--sphere}
%\date{\today}
\author{K. Leschke}
\address{Katrin Leschke\\ Department of Mathematics\\University of Leicester\\University Road\\
  Leicester, LE1 7RH, UK}

\email{k.leschke@mcs.le.ac.uk}
\thanks{Author partially supported by DFG SPP 1154 ``Global Differential Geometry''}

\begin{abstract}
  For all positive integers $n$ we construct a 1--parameter family of
  conformal tori of revolution in the 3--sphere with
  $n$ bulges. These tori arise by Darboux transformations of constant
  mean curvature tori in the 3--sphere but do not have constant mean
  curvature in $S^3$.
\end{abstract}
%%% Local Variables: 
%%% mode: latex
%%% TeX-master: "doc"
%%% End: 

\maketitle

\section{Introduction}

In a recent paper \cite{conformal_tori} it is shown that the
multiplier spectral curve of a conformal torus $f: T^2\to S^4$ is
essentially given by the set of closed Darboux transforms of $f$: to
each multiplier on the spectral curve there exists a quaternionic
holomorphic section with the given multiplier in the associated
quaternionic holomorphic line bundle $W$ of $f$. The prolongation of
the holomorphic section defines a new conformal torus $\hat f$, and it
turns out that $f$ and $\hat f$ satisfy a ``weak enveloping''
condition. Thus $(f, \hat f)$ form a generalized
Darboux pair: Classically, the Darboux transformation is defined for
isothermic surfaces and a map $\hat f: M \to \R^3$ is called a
\emph{classical Darboux transform} \cite{darboux} of an isothermic $f:
M \to \R^3$ if there exists a sphere congruence enveloping both $f$
and $\hat f$.

For every conformal torus the set $H^0_h(W)$ of holomorphic sections
with a given multiplier $h$ is generically 1--dimensional, and at
generic points the Darboux transformation preserves geometric
properties: e.g., generic Darboux transforms of a constant mean curvature
torus have constant mean curvature \cite{cmc_tori}, and
generic Darboux transforms of a Hamiltonian stationary torus are  
Hamiltonian stationary \cite{hsl}.

However, there exist examples, e.g. \cite{hsl}, of conformal tori which allow
non--trivial multiplier on the spectral curve with high dimensional
space of holomorphic sections. The existence of these singular
multipliers should allow a deformation of the spectral curve: in the
case of constant mean curvature tori in the 3--sphere of spectral genus zero
\cite{lawson_conjecture} one can deform the spectral curve to obtain a
family of Delaunay tori by removing this singularity of the spectral
curve, and thus by adding geometric genus. By contrast the Darboux
transformation preserves the geometric spectral genus
\cite{conformal_tori} in the case when the Darboux transform is
immersed but it may change geometric properties (e.g. break the
constant mean curvature condition). In particular, the Darboux
transformation at singular points is expected to allow to add or
remove arithmetic genus of the spectral curve, and a thorough
understanding of the singular points of the multiplier spectral curve
may play an important role in understanding the reconstruction of
conformal tori by their spectral data \cite{hitchin, pin&ster, Bob,
  QKP} and the study of minimum energy tori in presence of a
variational principle \cite{schmidt, lawson_conjecture}.

In this short note, we concentrate on the geometric properties of the
Darboux transformation in the case when the conformal torus is a
rectangular torus in the 3--sphere: in particular, we construct for
each $n\in \N, n\ge 2,$ a 1--parameter family of conformal tori of
revolution in $S^3$ with $n$ bulges which do not have constant mean
curvature in $S^3$.  Using a similar argument we also construct for
each $a\in\R, a>0,$ a 1--parameter family of cylinder of revolution
with non--constant mean curvature.

{\bf Acknowledgments.}  The author would like to thank
U. Hertrich--Jeromin, F. Pedit and N. Schmitt for fruitful discussions
during the preparation of this work.

%%% Local Variables: 
%%% mode: latex
%%% TeX-master: "doc"
%%% End: 

%\input{holomorphic}

\section{Rectangular tori}

In the following, we will apply the Darboux transformation on
rectangular tori $f: \C/\Gamma \to S^3 \subset\R^4$ with lattice
$\Gamma = \frac{1}{u} \Z \oplus \frac{i}{v} \Z$, $u, v>0$, to
construct the families of conformal tori of revolution. We identify
Euclidean 4--space $\R^4=\H$ with the quaternions and parametrize a
\emph{rectangular torus with parameters} $(u,v)$ by
\[
f(x,y) = \frac{uv}{u^2+v^2}\left( \frac{1}{u} e^{2\pi j u x} + i \frac 1{v} e^{2\pi j v y}\right)\,.
\]
In particular, we will use the fact that a rectangular torus $f:
T^2\to S^3$ is Hamiltonian stationary, and use the methods and
settings developed in \cite{hsl} to compute the Darboux transforms of
$f$.  We can write $f$ as
\[
f(x,y) =  \ejbh\left(\frac 1{u} + i \frac 1{v}\right)\left(-\frac j{2\pi}\right)g\,,
\]
where 
the
so--called \emph{Lagrangian angle} $\beta$ is given by
\[
\beta(z)= 2\pi\langle \beta_0, z\rangle \quad \text{ with } \quad
\beta_0 = u-v i\in\Gamma^* = u\Z \oplus v \Z\,,
\]
and
\[
g= 2\pi \rho j e^{\pi j(u x+v y)}\,, 
\]

with scale $\rho = \frac{uv}{\sqrt{u^2 + v^2}}\in\R$. Moreover, 
The derivative of $f$ can be written as
\[
df =  2\pi \rho e^{\pi j(u x- v y)}\, dz\, j e^{\pi j(u x+v y)} = \ejbh dz g\,,
\]
and $f$ is a conformal immersion, that is \cite[Sec. 2.2]{coimbra}
\[
*df = Ndf = -df R\,.
\]
with left normal
\[
N = e^{j\beta}i = e^{2\pi j(ux - vy)}i
\]
and right normal
\[
R = -g\invers i g= i e^{2\pi j(ux + vy)}\,.
\]
For every conformal immersion $f: M \to \H$ with right and left
normals $R$ and $N$ the normal bundle of $f$ is given \cite[Sec 2.2]{coimbra}
by
\[
\perp_f = \{ x\in\H \mid N x R = - x\}\,.
\]
In particular, if $f: M\to S^3\subset\H$ is a conformal map into the
3--sphere, then $n=Nf = fR$ and $f$ are unit normals. Thus, the second
fundamental form $\II_\H$ of $f$ as a map into $\H$ computes with
\[
(X df(Y))^\perp= -<df(Y), dn(X)> n - <df(Y), df(X)>f
\] 
as
\[
\II_\H = \II_{S^3} - |df|^2f\,.
\]
From this we see that the mean curvature vector $\Hh_\H$ in $\R^4=\H$
relates to the mean curvature vector $\Hh_{S^3} = H_{S^3}n$ via
\[
\Hh_\H = H_{S^3} n - f
\]
where $H_{S^3}$ is the mean curvature of $f$ in $S^3$. We denote by
\[
(dN)' = \frac 12(dN - N*dN) \quad \text{ and } \quad (dN)'' = \frac 12(dN + N *dN)
\]
the $(1,0)$ and $(0,1)$--parts of the derivative of $N$ with respect
to the complex structure $N$, and define $H$ by $(dN)' = -df H$.  Then
it is shown in \cite[Sec. 7.2]{coimbra} that the mean curvature vector
of a conformal immersion into $\R^4$ is given by
\begin{equation}
\label{eq:meancurvaturevector}
\Hh_\H = N\bar H\,.
\end{equation}
Combining the previous equations, we see that the mean
curvature $H_{S^3}$ of a conformal immersion $f: M \to S^3$ of a
Riemann surface $M$ into $S^3$ is given by
\begin{equation}
\label{eq:H in S3}
H_{S^3} = f H + N =\Re(fH)\,.
\end{equation}
In particular, since $H$ computes in the case of Hamiltonian
stationary Lagrangians \cite{hsl} to
\[
H = \pi g\invers \bar \beta_0 e^{j\frac{\beta_0}2} k
\]
the constant mean curvature of a rectangular torus in $S^3$ is given
by 
\[
H_{S^3} = \frac 12(\frac uv - \frac vu)\,.
\]
 
%%% Local Variables: 
%%% mode: latex
%%% TeX-master: "doc"
%%% End: 

\section{The Darboux transformation}
\label{sec:Darboux}

We will briefly recall the construction of Darboux transforms in the
case when $f: M \to \R^4$ is a conformal immersion from a Riemann
surface into Euclidean 4--space. For the general case of conformal
immersions into the 4--sphere and details of the construction compare
\cite{conformal_tori}.  In our situation, the associated quaternionic
holomorphic line bundle of the immersion $f$ can be identified with
the trivial quaternionic bundle  $\trivial{}$ equipped with the (quaternionic)
\emph{holomorphic structure} $D$ given by
\[ 
D \alpha := \frac 12(d\alpha + N*d\alpha)\,,
\]
for $\alpha\in\Gamma(\trivial{})$ where $N$ is the left normal of $f:
M \to \R^4=\H$. We denote by $H^0(\trivial{}) = \ker D$ the set of
holomorphic sections of the holomorphic line bundle $(\trivial{}, D)$.  The
\emph{prolongation} of a local holomorphic section $\alpha\in
H^0(\trivial{})$ is given by the local section
\begin{equation}
\label{eq:prolongation}
\psi = \begin{pmatrix} f \nu + \alpha\\ \nu
\end{pmatrix} \in\Gamma(\trivial{2})
\end{equation}
of the trivial $\H^2$ bundle where $\nu$ is defined by $d\alpha = - df
\nu$. Then $\psi$ spans locally a quaternionic line bundle $\hat L$,
and if $d\alpha$ is nowhere vanishing, the corresponding map $\hat f =
f + \alpha\nu\invers$ is a branched conformal immersion into $\R^4$, a
so--called a \emph{Darboux transform} of $f$.  If we denote by $T =
\hat f- f$, then the derivative of $\hat f$ is given away from the
zeros of $\alpha$ by
\[
d\hat f = - T d\nu\alpha\invers T\,.
\]
From $d\alpha = - df\nu$ we see that $df\wedge d\nu=0$, in other
words, $*d\nu = - R d\nu$ since $f$ has right normal $R$. In
particular, $\hat f$ has left normal
\begin{equation}
\label{eq:hatN}
\hat N = -T  R T\invers\,.
\end{equation}
 To compute the mean curvature
vector $\hat\Hh$ of the Darboux transform $\hat f$, it remains
(\ref{eq:meancurvaturevector}) to compute $\hat H$ using the defining
equation $(d\hat N)' = - d \hat f\hat H$. To this end, note that the
derivative of $\hat N$ computes with $\hat f = f +T$ as
\[
d\hat N = -*dfT\invers + d\hat f T\invers\hat N - TdR T\invers + \hat
N df T\invers -* d\hat f T\invers
\]
so that the $(1,0)$-part of $d\hat N$ with respect to $\hat N$ is
given by
\begin{equation}
\label{eq:dhatN10}
(d\hat N)' = d\hat f T\invers \hat N - T(dR)''T\invers - *d\hat f T\invers
\end{equation}
where $(dR)'' = \frac 12(dR +R*dR)$.

To obtain Darboux transforms which are globally defined, we consider
holomorphic sections with multiplier, that is holomorphic sections of
the trivial bundle $\tilde{\trivial{}}$ over the universal cover
$\tilde M$ of $M$ which satisfy
\[
\gamma^*\alpha = \alpha h_\gamma
\]
with $h_\gamma\in\C_*$ for all $\gamma\in \pi_1(M)$. From $d\alpha =
-df \nu$ and (\ref{eq:prolongation}) we see that the prolongation
$\psi$ of $\alpha$ has multiplier $h_\gamma$, that is $\gamma^*\psi =
\psi h_\gamma$ for $\gamma\in\pi_1(M)$, so that $\psi$ defines a
branched conformal immersion $\hat f: M \to \H$ if $\alpha$ is nowhere
vanishing.

In the case when $f: T^2 \to S^4$ is a conformal 2--torus, the
existence of global Darboux transforms is guaranteed by the link
\cite{conformal_tori} between Darboux transforms and the
\emph{multiplier spectral curve} $\Sigma$ of $f$: to every multiplier
$h\in\Sigma$ there exists at least one holomorphic section with
multiplier $h$, and each such holomorphic section gives by
prolongation a Darboux transform $\hat f: T^2 \to S^4$ of $f$. In
other words, there is at least a Riemann surface worth of Darboux
transforms of a conformal torus.

%%% Local Variables: 
%%% mode: latex
%%% TeX-master: "doc"
%%% End: 
\section{Darboux transforms of Hamiltonian stationary Lagrangian tori}

In the following we summarize notations and results of \cite{hsl}. In
the case of an Hamiltonian stationary Lagrangian torus $f:
\C/\Gamma\to \R^4$ with Lagrangian angle $\beta =2\pi \langle \beta_0,
\cdot \rangle$, every multiplier of a holomorphic section is of the
form
\[
h = h^{A,B} = e^{2\pi(\langle A, \cdot\rangle -i\langle B, \cdot
  \rangle)}
\]
with $A,B\in\C^2$ such that 
\[
\Gamma^*_{A,B} = \{\delta\in\Gamma^*+\frac{\beta_0}2 \mid \delta \text{ satisfies }
  | \delta - B|^2 - |A|^2 = \frac{| \beta_0 |^2}{4}   
\text{ and }  \langle \delta-B, A\rangle =0
\}
\]
is not empty.  A holomorphic section $\alpha\in H^0(\widetilde{V/L})$
with multiplier $h^{A,B}$ is called \emph{monochromatic} if it is
given by a Fourier monomial, that is if  
\begin{equation}
\label{eq:monochromatic holomorphic section}
\alpha = \alpha_\delta:=\ejbh (1-k\lambda_\delta)  e_{\delta-B} e^{2\pi\langle A, \cdot\rangle}
\end{equation}
is given by a single frequency $\delta\in\Gamma^*_{A,B}$ where $
\lambda_\delta \assign \frac{2}{\beta_0}(\delta-i A - B)$ and
\[
e_\gamma(z):= e^{2\pi i
  \langle\gamma,z\rangle}\,.
\]
A \emph{polychromatic} holomorphic section with multiplier $h^{A,B}$
is given by a non--trivial linear combination $\alpha
=\sum_{\delta\in\Gamma^*_{A,B}} m_\delta \alpha_\delta$ of
monochromatic holomorphic sections, $m_\alpha\in\C$.

\begin{definition}
  A branched conformal immersion $\hat f: M \to S^4$ is called a
  \emph{monochromatic Darboux transform} (respectively
  \emph{polychromatic}) if it is given by the prolongation of a
  monochromatic (respectively polychromatic) holomorphic section.
\end{definition}

In \cite{hsl} it is shown that all monochromatic Darboux transforms of
a rectangular torus are after reparametrization again a rectangular
torus. Moreover, polychromatic holomorphic sections with multiplier
$h^{A,B}, A\not=0$, only occur if $h^{A,B}$ is real. In particular,
the corresponding Darboux transform coincides with a monochromatic
Darboux transform. To obtain new tori we therefore have to consider
polychromatic Darboux transforms with $A=0$:

\begin{theorem}[\cite{hsl}]
\label{thm:Darboux transforms}
Let $f: \C/\Gamma\to\R^4$ be a Hamiltonian stationary torus in $\R^4$
with Lagrangian angle $\beta$ and $df = \ejbh dz g$. Then every
non--constant, polychromatic Darboux transform $\hat f: \C/\Gamma\to\R^4$
 of  $f$  with $A=0$ is given by
\begin{equation}
\label{eq: multiple darboux}
\hat f = f + \ejbh \left(\sum_{s,t\in I_B} (1+k e^{is}) m_s \bar m_t e_{\frac{\beta_0}2(e^{it} -e^{is})} (1 + ke^{it})  \sin t\right) \frac{1}{R \pi \bar\beta_0} g
\end{equation}
where   the finite set
\[
I_B=\{ t\in [0, 2\pi) \mid B-\frac{\beta_0}2e^{it}\in\Gamma^*_{0,B}\}\not=\{0,\pi\}
\]
parametrizes the admissible frequencies and $m_t\in \C$ are chosen so
that the map
\[
R =|\sum_{t\in I_B} m_t \sin t \, e_{B-\frac{\beta_0}2e^{it}}|^2 +
|\sum_{t\in I_B} m_t e^{it}\sin t\, e_{B-\frac{\beta_0}2e^{it}}|^2
\]
is nowhere vanishing.   
\end{theorem}

\begin{comment}

since $T$ satisfies a Riccati type equation
\[
dT = -df  - T\omega T
\]
with $*\omega = -R \omega$. (We know more: $\omega =d\nu \alpha\invers$ with $\nu=T\invers\alpha$.) In particular, the $(1,0)$--part of $d\hat N$ with respect to the complex structure given by left multiplication by $\hat N$ is given by
\[
(d\hat N)' = - T(\omega \hat N - *\omega +(dR)''  T\invers)
\]
where $(dR)'' = \frac 12(dR +R*dR)$.
\begin{krem}
verify signs (seems okay). test the results by computing the mean curvature of a monochromatic Darboux transform.
\end{krem}
 Since $(dR)''$ and $\omega$ have
the same left type $-R$ there exists $\lambda: M \to \H$ with
\[
\omega =(dR)''\lambda
\]
(at least away from the zeros of $(dR)''$). Therefore, we get
 with $d\hat f = -T\omega T$
\begin{equation}
\label{eq:hat H}
\hat H =  T\invers(-\hat N + \lambda\invers R\lambda - \lambda\invers T\invers)
\,.
\end{equation}

\begin{krem} We use this rather ugly description rather than to write
  everything in terms of $\omega$ since it seems, though ugly, the
  easiest to compute version of $\hat H$. Need to double check the
  current expression for $\hat H$ since it was computed in a hurry.
\end{krem}

\end{comment}
%%% Local Variables: 
%%% mode: latex
%%% TeX-master: "doc"
%%% End: 

\section{Families of conformal tori of revolution in $S^3$}

In this section we discuss the Darboux transforms of a rectangular
torus $f: T^2\to S^3$. If the parameters $(u,v)$ of $f$ satisfy $u \ge
\sqrt{3} v$, we show that there exist polychromatic Darboux transforms
which are conformal tori in the 3--sphere. 

Fix $n\in\N, n\ge 2,$ and consider all rectangular tori $f$ in the
3--sphere with parameters $(u,v)$ satisfying $u^2 + v^2(1-n^2)\ge 0$. 
The multiplier $h=h^{0,B}$ with
\[
B=\frac{\beta_0}2
+\frac{nv i}2 - \frac{\sqrt{u^2 + v^2(1-n^2)}}2\,.
\] 
has $\dim H^0_{0,B} = |\Gamma^*_{0,B}|\ge 2$ since
\[
\delta_+ =\frac{\beta_0}2, \quad \delta_- =\frac{\beta_0}2 + nvi \in\Gamma^*_{0,B} =\{ \delta\in\Gamma^2 +\frac{\beta_0}2 \mid |\delta-B|^2 =\frac{|\beta_0|^2}4\}\,,
\]
which shows that $f$ allows polychromatic Darboux transforms. In
particular, we see that
 \[
-\lambda_{\delta_\pm} = \frac{2}{\beta_0}(B-\delta_\pm) = 
c_\pm + i s_\pm\]
 with
\begin{equation}
\label{eq:cspm}
c_\pm = \frac{\mp nv^2 - us}{r^2}, \quad \quad  s_\pm = \frac{\pm uvn -sv}{r^2} 
\end{equation}
and
\[
s :=  \sqrt{r^2 -v^2n^2}
\quad \text{ where } \quad
r = |\beta_0|= \sqrt{ u^2 + v^2}
\]
so that \eqref{eq:monochromatic holomorphic section} all monochromatic holomorphic sections with multiplier $h^{0,B}$ are given by 
\[
\alpha_{\pm} := \alpha_{\delta_\pm } =  e^{\pi j(u x -vy)}\left(1 + j s_\pm + k c_\pm\right) e^{\pi i(s x \mp nv y)}\,.
\]
Using Theorem \ref{thm:Darboux transforms} we see that the
corresponding polychromatic Darboux transforms are given by
\[
\hat f = \ejbh\left(-\left(\frac j{u} + \frac k{v}\right)
 + (\sum_{i,j\in\{\pm\}} a_{ij}) \frac{u - iv}{R r^2} \right) \frac{g}{2\pi} 
\]
where   
\begin{equation}
\label{eq:aij}
a_{ij} =  (1+k e^{it_i}) m_i \bar m_j e^{2\pi i\langle(\delta_i-\delta_j),\cdot\rangle} (1 + ke^{it_j})  \sin t_j\,,
\end{equation}
and $m_\pm\in\C$ have to be chosen so that
\begin{eqnarray*}
2R(y) &=& |m_+s_+ + m_-s_-e^{i\tilde y}|^2  + \left|m_+s_+e^{it_+} + 
m_- s_-e^{it_-}e^{i\tilde y}\right|^2
\end{eqnarray*}
is nowhere vanishing where 
\[
\tilde y := 2\pi nvy\,.
\]
Using (\ref{eq:cspm}) and $|z+w|^2 = |z|^2 +|w|^2 + 2\Re(z\bar w)$ the
denominatior $R$ simplifies to
\begin{eqnarray*}
R &=& |m_+|^2s_+^2 + |m_-|^2 s_-^2  + \frac{2v^2(1-n^2)}{r^4} \Re\left(m_+\bar m_- (r^2- n^2v^2- is vn)e^{-i\tilde y}\right)\,.
\end{eqnarray*}

In general, the above Darboux transforms will be conformal immersions
into the 4--sphere. However, there exist constants $m_+, m_-\in\C_*$
so that the polychromatic Darboux transform is a conformal immersion
into the 3--sphere.

\begin{theorem}
  For each $n\in\N, n\ge 2,$ there exists a 1--parameter family of
  conformal tori of revolution in the 3--sphere with $n$ bulges. Only
  one conformal torus in this family has constant mean curvature in
  $S^3$.
\end{theorem}
\begin{proof}
We consider the case that  $m_\pm = m$. In this case (\ref{eq:cspm})
\[
s_+^2 + s_-^2 = \frac{2v^2}{r^4} \left(u^2(1 + n^2) + v^2(1-n^2)\right)
\]
shows that
\[
\frac{Rr^4}{2|m|^2v^2}=  %u^2(1+n^2) + v^2(1-n^2) + (1-n^2)\Re\left((r^2-n^2v^2 -isvn)e^{-i\tilde y}\right)\,.
u^2(1+n^2) + v^2(1-n^2) + (1-n^2)\left((r^2-n^2v^2)\cos(\tilde y) -svn\sin(\tilde y) \right)\,.
\]
From (\ref{eq:aij})
\[ 
a_{ij} =  |m|^2(1+k e^{it_i})  e^{2\pi i\langle(\delta_i-\delta_j),\cdot\rangle} (1 + ke^{it_j})  \sin t_j
\]
we see that $\hat f$ is independent of the choice of $m\in\C$. In particular,  we may assume from now on that  $m=1$.  Then
\[
a_{\pm, \pm } = 2ke^{it_\pm}s_\pm 
= 2(k c_\pm s_\pm + j s_\pm ^2)
\]
is independent of $z= x + iy$ while
\[
a_{\pm, \mp} = (1 + ke^{it_\pm})s_\mp e^{\mp i \tilde y} + 
(k e^{it_\mp} - e^{it_\mp - i t_\pm}) s_\mp e^{\pm i \tilde y}
\]
is independent of $x$.  In particular, the Darboux transform is given
by $\hat f = f + T$ where $T= \ejbh \tau g$, $\tau=\tau_0 + i \tau_1$
with
\begin{eqnarray*}
\tau_0(y) &=& \frac{jun^2}{\pi \hat R(y)}\in j\R\\
\tau_1(y) &=&\frac{ n(snv \cos\tilde y + s^2 \sin \tilde y)+  j(s^2+(s^2\cos\tilde y - svn \sin \tilde y))}{\pi v\hat R(y)}  \in\Cj = \Span\{1, j\}
\end{eqnarray*}
  and 
\[
\hat R(y) =  u^2(1+n^2) + v^2(1-n^2) + (1-n^2)(s^2\cos\tilde y - svn \sin \tilde y)\,.
\]

Writing $\hat f = \ejbh \hat \tau g$ where
\begin{equation}
\label{eq:hattau}
\hat \tau = \tau + \sigma, \quad \quad \sigma = -\frac{1}{2\pi}\left(\frac{j}{u} + \frac{k}{v}\right)\,,
\end{equation}
a lengthy, but straightforward, computation shows that 
\begin{equation}
\label{eq:length tau}
|\hat \tau|^2  = \frac{1}{4\pi^2\rho^2}\,.
\end{equation}
In other words, $\hat f: T^2\to S^3$ is a conformal map into the
3--sphere. On the other hand, we have
\[
\hat f = e^{j\pi(ux - vy)} (\tau_0 -\frac{j}{2\pi u} + i(\tau_1- \frac{j}{2\pi v})) 2\pi \rho j e^{j\pi(ux + vy)}
  = e^{2j\pi ux} \kappa_0 + i e^{2j\pi vy} \kappa_1
\]
where $\kappa_0= (2\pi\tau_0j + \frac 1{ u}) \rho $ is a real valued
function, and both $\kappa_0$ and $\kappa_1 = (2\pi\tau_1 j + \frac 1{
  v}) \rho$ only depend on $y$, that is, $\hat f$ is a surface
of revolution in $S^3$. Note that for $u = v\sqrt{n^2-1}$, the Darboux
transform $\hat f$ is a rectangular torus with constant mean curvature
in the 3--sphere. If $u\not= v\sqrt{n^2-1}$ then $\kappa_0$ is
extremal at
\[
y_k  = \frac 1{2nv}\left(\frac 1\pi \arctan(-\frac{vn}{\sqrt{u^2+v^2(1-n^2)}}) + k\right), \quad k=0,\ldots, 2n-1\,,
\]
in particular, $\hat f$ is a torus of revolution in $S^3$ with $n$
bulges.  
\begin{figure}
\includegraphics[height=4.2cm]{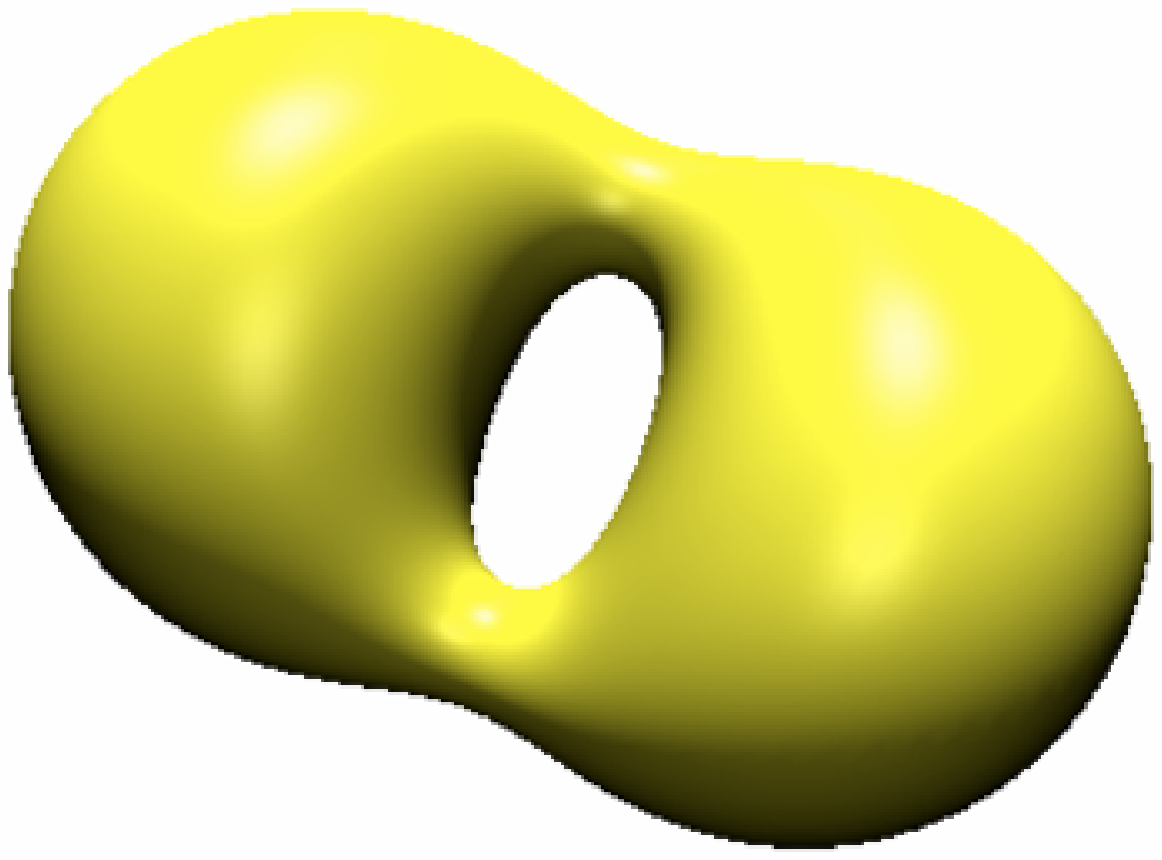}
\includegraphics[height=4.2cm]{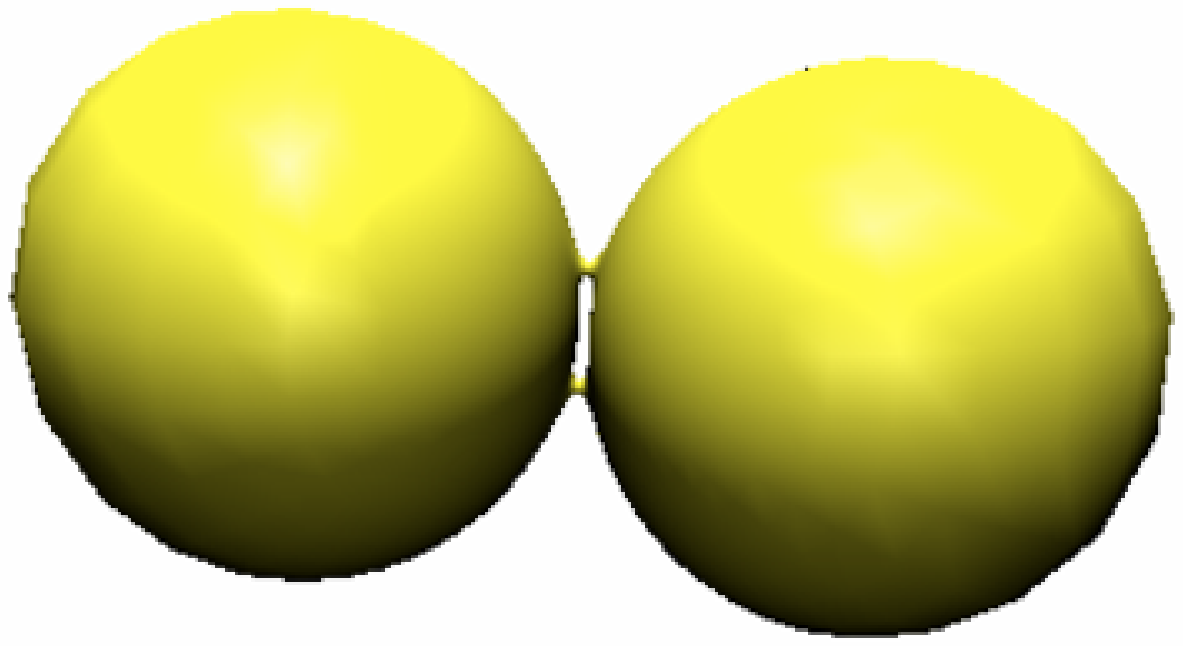}
\caption{Darboux transforms of rectangular tori with parameter (1.8,1) and (2.1, 1)}
\end{figure}
If $\hat f$ had constant mean curvature in $S^3$ and is not a
rectangular torus, we would $\hat f$ expect to be a Delaunay torus and
thus to be parametrized by elliptic functions. However, the Darboux
transformation essentially preserves spectral genus
\cite{conformal_tori}, and in our case $\hat f$ is parametrized by
(rational functions of) trigonometric functions. Indeed, we compute
the mean curvature $\hat H_{S^3} = \Re(\hat f\hat H)$ of $\hat f $ in
$S^3$ explicitly when $u\not=v\sqrt{n^2-1}$: first, $\hat H$ is
defined by $(d\hat N)' = - d\hat f \hat H$, that is (\ref{eq:dhatN10})
\[
\hat H 
= -T\invers\hat N + \hat f_x\invers(Tr_x + \hat N \hat f_x) T\invers
\]  
where $(dR)'' = r_x dx + r_y dy$ with
\[
r_x = \pi g\invers j(ui-v) g\,.
\]
\begin{figure}
\includegraphics[height=4.2cm]{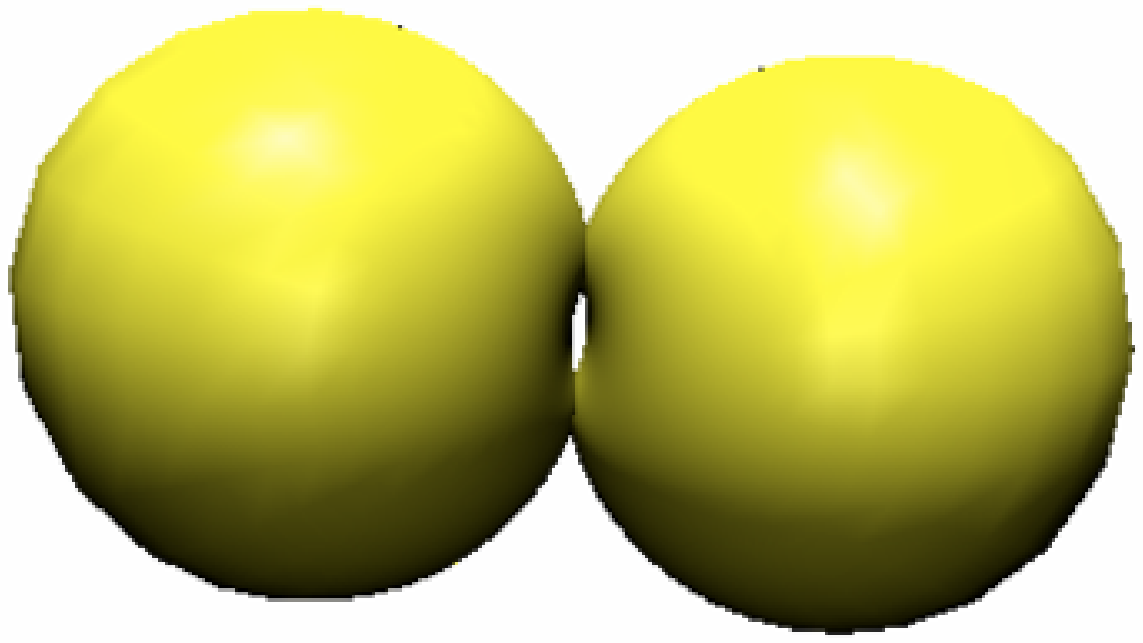}
\includegraphics[height=4.2cm]{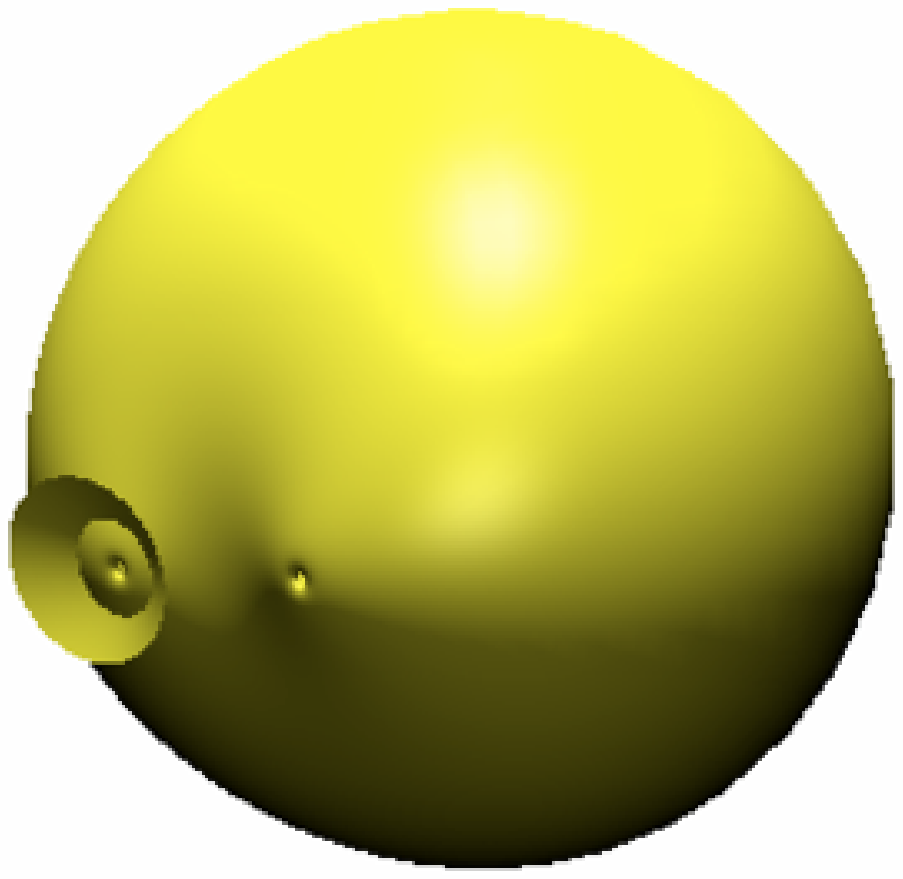}
\caption{Non--embedded Darboux transform of the rectangular torus with
  parameter (2.6, 1)}
\end{figure}
Moreover, we compute $\hat f_x = \ejbh 2\pi ju \hat\tau_0 g = \ejbh q
g$ with real valued
\[
q=2\pi ju\hat\tau_0 =1 -\frac{2u^2n^2}{\hat R}\,.
\]
The left normal of $\hat f$ is given by (\ref{eq:hatN})
 \[
 \hat N = \ejbh \tau i \tau\invers \emjbh 
\] 
so that
$ \hat f\hat H = \ejbh \hat\gamma \emjbh $ with $ \hat\gamma =
\hat\tau\gamma \tau\invers $ and
\[
\gamma = -i+ q\invers\pi \tau j (ui-v) + \tau i \tau\invers \,.
\]
Next observe that for any $z, w\in\H$ we have $\Re(z w z\invers) =\Re w$ which shows that 
\[
\Re(\gamma)  = q\invers \pi
\Im(v\tau_0 + u \tau_1)\,,
\]
and the mean curvature of $\hat f$ in $S^3$ is, using $\hat\tau = \tau
+ \sigma$, given by (\ref{eq:H in S3})
\begin{equation}
\label{eq:hat H}
\hat H_{S^3} = \Re(\sigma\gamma\tau\invers) + \frac \pi q \Im(v\tau_0 + u \tau_1)\,,
\end{equation}
Furthermore, for $\lambda =\lambda_0 + i \lambda_1\in \H, \lambda_0, \lambda_1\in\Cj$, we have with
$\bar\tau = -\tau$
\[
 \Re(\sigma \lambda \tau\invers)=  -\frac 1{2\pi |\tau|^2}\Im\left(\tau_0\left(\frac{\lambda_0}u + \frac{\lambda_1} v\right)+ \tau_1\left(\frac{\bar\lambda_0}v - \frac{\bar\lambda_1}u\right)\right)
\]
and we get for both $\lambda=-i$ and $\lambda= \tau i\tau\invers $ 
\[
\Re(\sigma \lambda \tau\invers) =\frac 1{2\pi |\tau|^2}\Im\left(\frac{\tau_0} v - \frac{\tau_1} u\right)\,.
\]
whereas for $\lambda=  q\invers \pi \tau j(ui-v)$
\[
\Re(\sigma \lambda \tau\invers) =\frac{1}{2q}(\frac vu + \frac uv)- \frac 1{uvq |\tau|^2}(v\Im \tau_0 + u \Im\tau_1)^2\,.
\]
A straightforward computation shows that 
\[
\pi  uv |\tau|^2 = \Im(v\tau_0 + u\tau_1)
\]
so that (\ref{eq:hat H}) simplifies to
\begin{equation}
\label{eq:hat H simplified}
\hat H_{S^3} =
% \frac \pi q \Im(v\tau_0 + u \tau_1) + \frac 1{\pi |\tau|^2}\Im\left(\frac{\tau_0} v - \frac{\tau_1} u\right)  + \frac{1}{2q}(\frac vu + \frac uv)- \frac 1{uvq |\tau|^2}(\Im(v\tau_0 + u\tau_1))^2 = 
\frac 1{\pi |\tau|^2}\Im\left(\frac{\tau_0} v - \frac{\tau_1} u\right)  + \frac{1}{2q}(\frac vu + \frac uv)\,.
\end{equation}
\begin{figure}
\includegraphics[height=4.2cm]{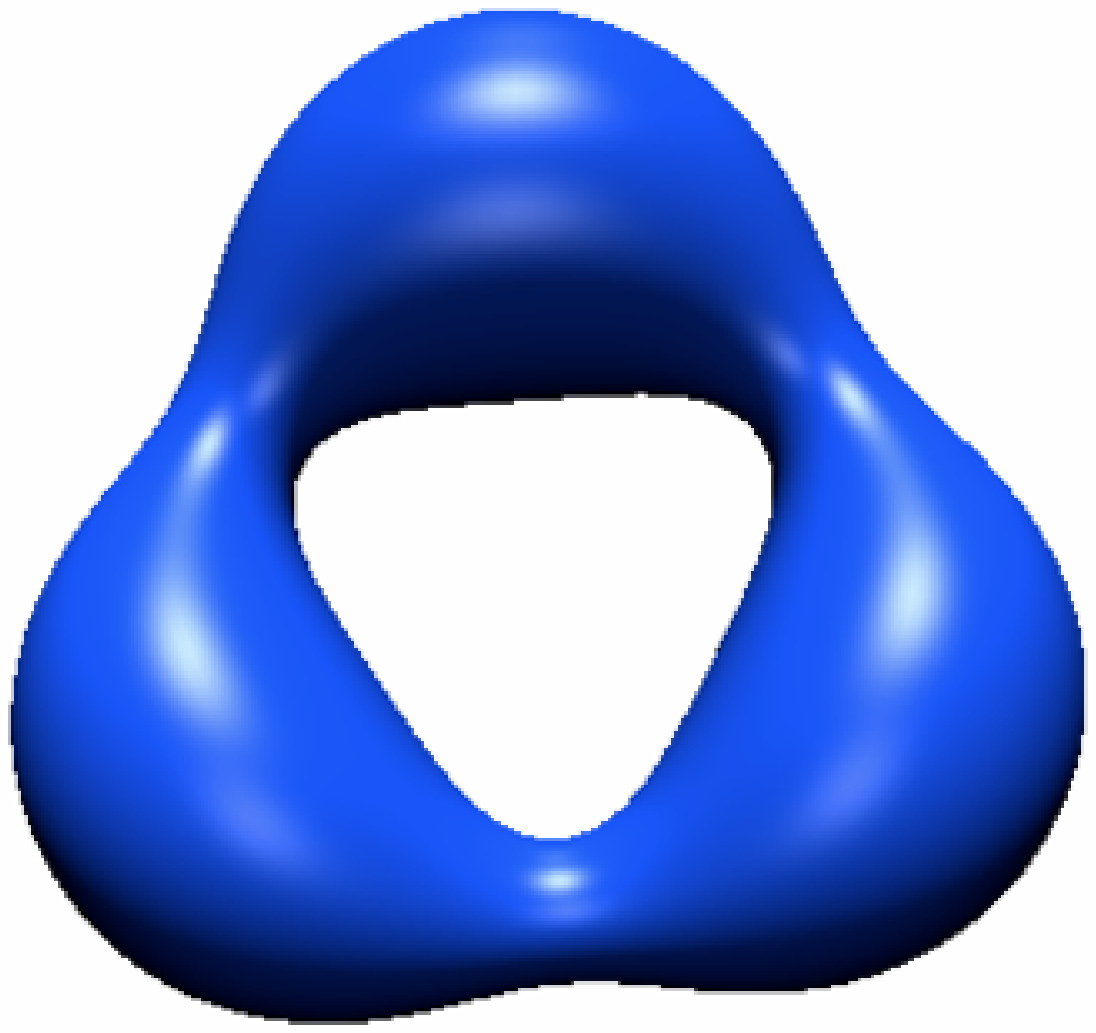}
\includegraphics[height=4.2cm]{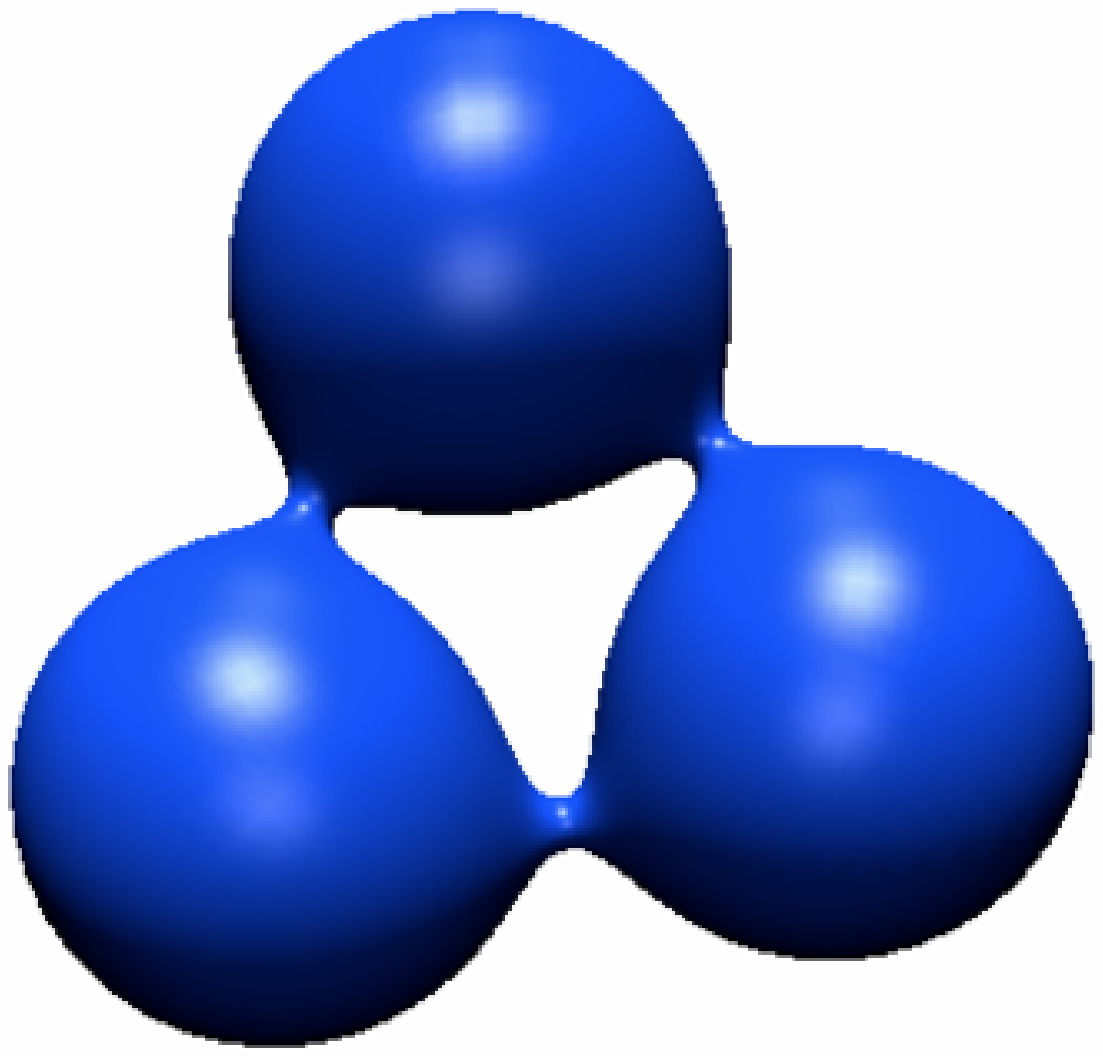}
\includegraphics[height=4.2cm]{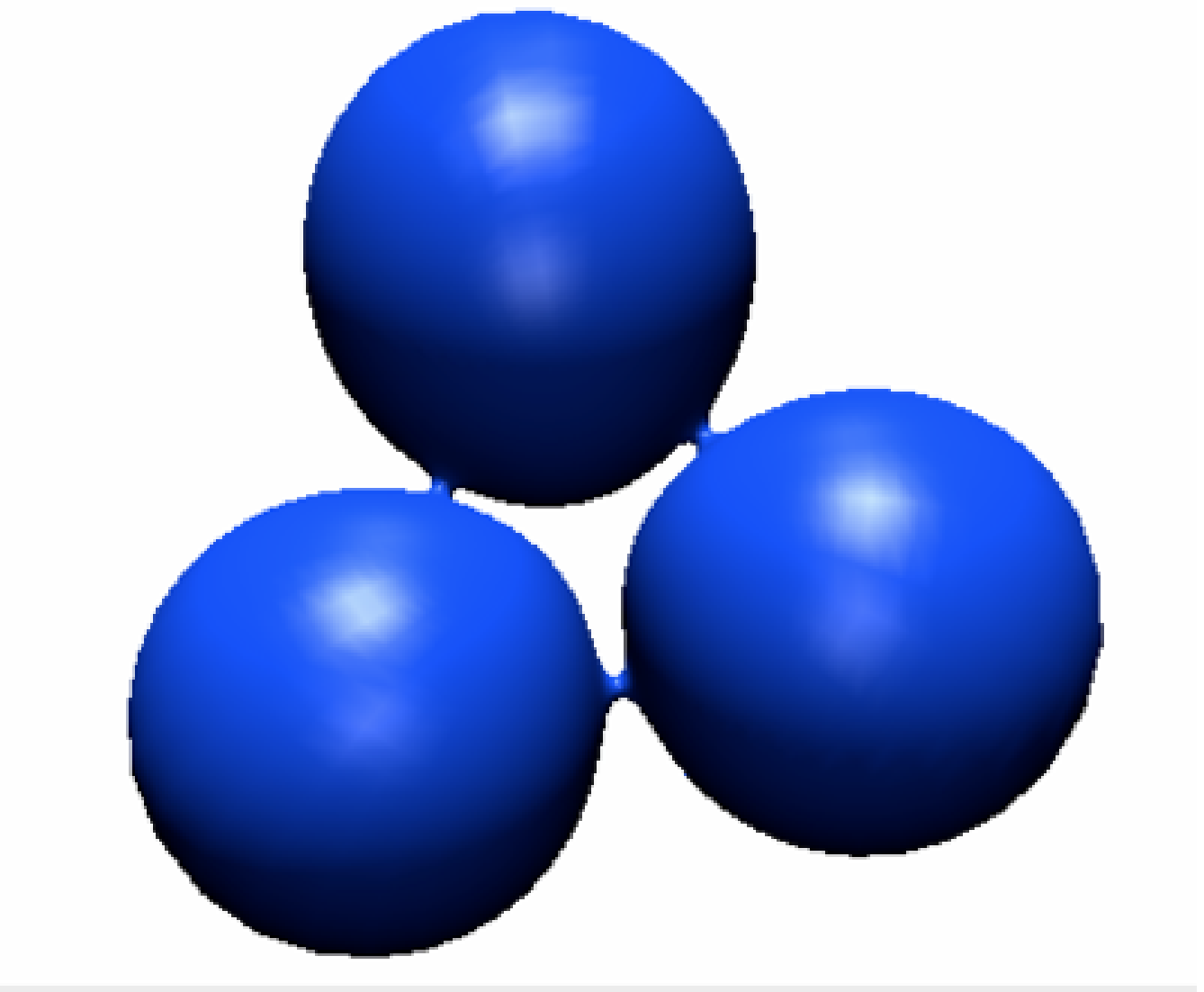}
\caption{Darboux transforms of rectangular tori with parameter (2.9,1), (3.2, 1) and (3.5,1)}
\end{figure}
For $y=0$ one easily obtains with 
\[
\tilde R  := 2u^2 + v^2(1-n^2)(2-n^2) \quad \text{ and } \quad \tilde q := 2r^2 -n^2v^2
\]
that
\[
\tau_0(0) =j \frac{n^2u}{\pi \tilde R}, \quad \text{ and }
\quad\tau_1(0) = \frac{1}{\pi v\tilde R}\left(vn^2 s + i(\tilde q -
  n^2v^2)\right)
\]
so that (\ref{eq:hat H simplified}) becomes
\[
\hat H_{S^3}(0) = % \frac vu\frac{r^2n^2-\tilde q}{\tilde q} - \frac{r^2\hat R}{2(n^2-1)uv \tilde q}=
%\frac{r^2\left(2v^2n^2(n^2-1)-\hat R(0)\right) -2\tilde q(n^2-1)v^2}{2(n^2-1)uv \tilde q} =
\frac{r^2(v^2n^4-\tilde q) -2\tilde q(n^2-1)v^2}{2(n^2-1)uv \tilde q}\,.
\]

\newcommand{\yt}{\frac{1}{2n}}
Similarly, for $y=\yt$ we obtain with
\[
 \tilde R := 2u^2 + v^2(1-n^2)
\]
that
\[
\tau_0 = j\frac{u}{\pi \tilde R}, \quad \text{ and } \quad \tau_1 = -\frac{s}{\pi \tilde R}
\]
 which gives 
\begin{comment}
\[
\frac{1}{\pi|\tau(\yt)|^2}\Im(\frac{\tau_0(\yt)}v - \frac{\tau_1(\yt)}u)
=\frac uv
\]
and
\[
\frac 1{2q(\yt)}(\frac uv + \frac vu)  = -\frac{r^2\tilde R}{1(n^2-1)uv^3}
\]
so that
\end{comment}
\[
\hat H_{S^3}(\yt) =
\frac{2u^2v^2(n^2-1) -r^2\tilde R}{2(n^2-1)uv^3}  
% = \frac{-2u^4  +(3n^2-5)u^2v^2 +(n^2-1)v^4}{2(n^2-1)uv^3} 
\]
In particular,  if $\hat f$ has constant mean curvature in $S^3$ then $\hat H_{S^3}(0) = \hat H_{S^3}(\yt)$ gives %(note that $\tilde R + v^2 =\tilde q$)
\[
r^2(v^2n^4-\tilde q)v^2 -2\tilde q(n^2-1)v^4 = 2u^2v^2(n^2-1)\tilde q -r^2\tilde R\tilde q
\]
which is equivalent to
\[ 4r^2(u^2+ v^2(1-n^2))^2=0\,,
\]
that is, $u=v\sqrt{n^2-1}$.  Finally, we notice that all rectangular
tori are, up to reparametrization $\hat z = vz$, $z=x+iy$, rectangular
tori with parameter $(\frac uv, 1)$. Thus, we obtain for each $n\in\N,
n\not=1$, a 1--parameter family of tori of revolution in $S^3$ with
$n$ bulges, each torus given by a polychromatic Darboux transform of
a rectangular torus with parameter $(u,1), u \ge \sqrt{n^2-1}$.
\end{proof}

\begin{figure}
\includegraphics[height=4.2cm]{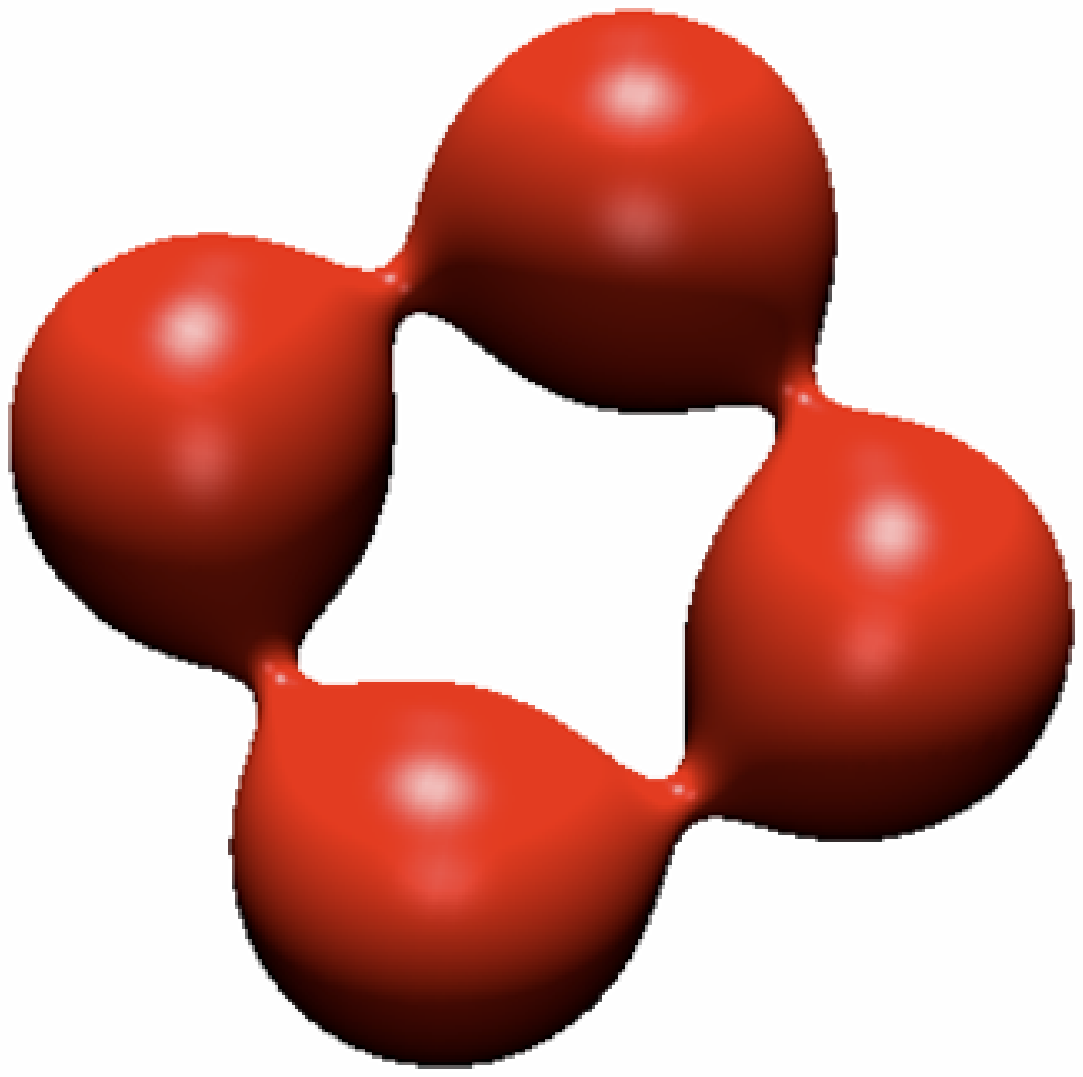}
\includegraphics[height=4.2cm]{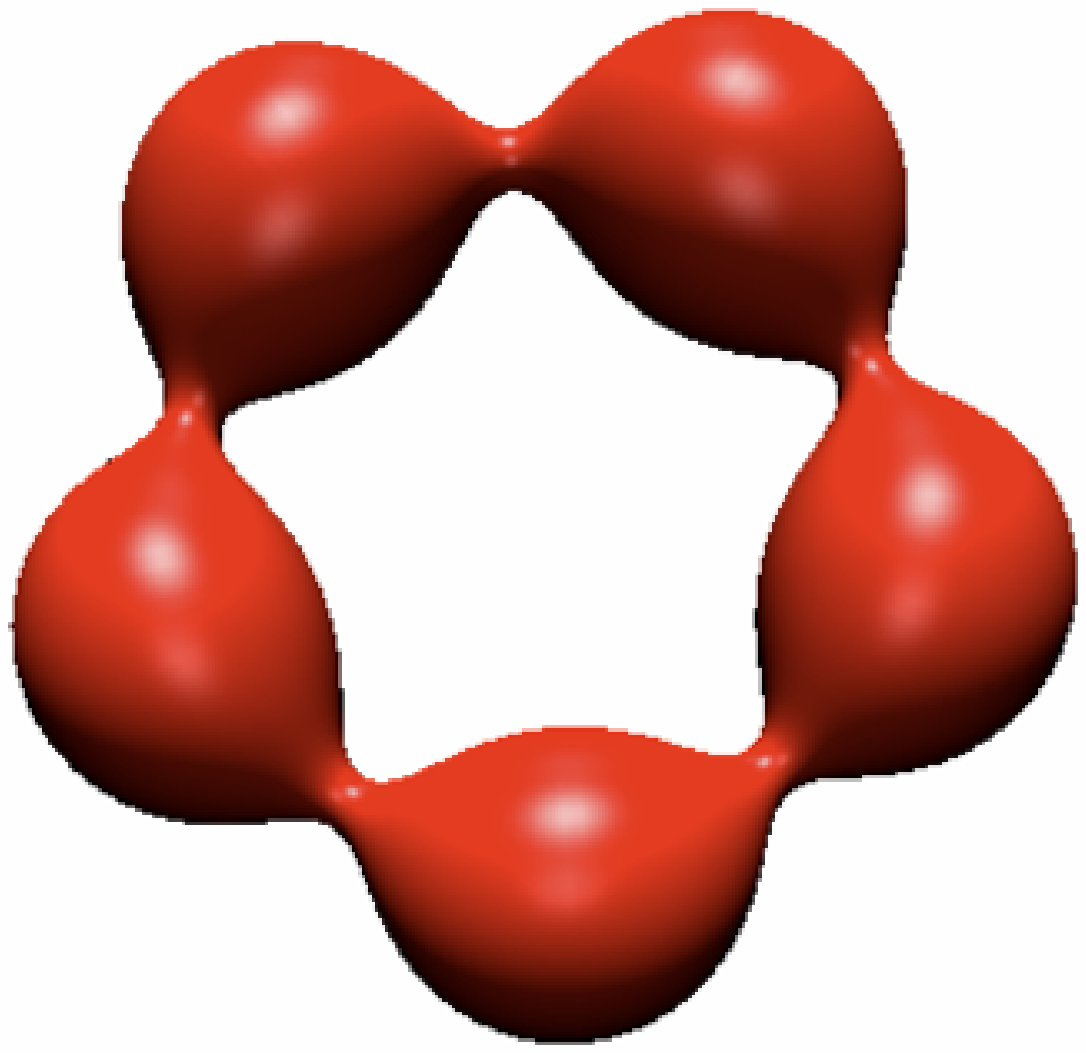}
\includegraphics[height=4.2cm]{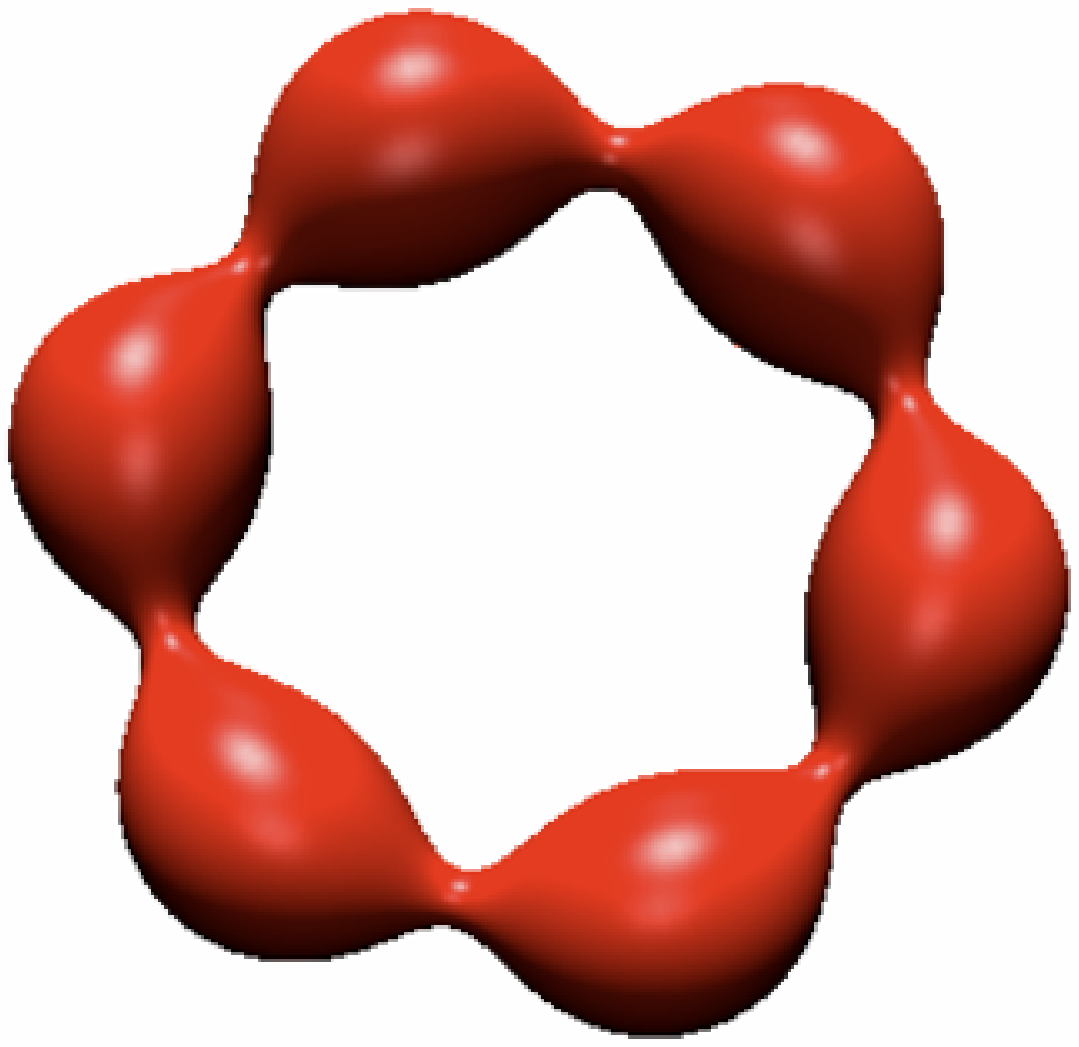}
\caption{Darboux transforms of rectangular tori with parameter (4.3,1), (5.3, 1) and (6.3,1)}
\end{figure}
 
\begin{comment}
 In particular, 
\[
\hat f_x = \ejbh g + (jT + Tj) \pi u\,.
\]
Thus, we obtain $\omega = -T\invers d\hat f T\invers = (dR)'' \lambda$ with
\[
\lambda = -r_x\invers T\invers \hat f_x T\invers
\]
where
 
\begin{rem}
  Can include the more general remark on $\mu$--Darboux transforms of
  CMC in $S^3$: in this case, my guess is that $\lambda = N(\hat a-1)
  + \hat b$ or some such thing, so that everything simplifies, and
  $\mu$--Darboux stays CMC in $S^3$. $\mu$--DT might be essentiell,
  though I still hope it's not in my particular case.
\end{rem}  

\end{comment}
%%% Local Variables: 
%%% mode: latex
%%% TeX-master: "doc"
%%% End: 

\section{Polychromatic Darboux transforms of cylinders}

We use similar methods to compute polychromatic Darboux transforms of
a standard cylinder $f: M \to \R^3$. To stay close to the notations
and computations in the previous sections, our maps $f$ will take
values in $\Span\{1, j, k\}$.  Note that in this case $f$ has mean
curvature \cite{coimbra} given by
\[
H_{\R^3} = - Hi
\]
where $H$ is again given by $(dN)' = -df H$. A standard cylinder 
\[
f(x,y) = \frac 1u  e^{2\pi j u x} + 2\pi k y
\]
is then a Hamiltonian stationary immersion with harmonic left and
right normals
\[
N = e^{2\pi j u x} i
 \quad \text{ and } \quad 
R = ie^{2\pi j ux}
\]
and Lagrangian angle $\beta(z) = 2\pi \langle\beta_0, z\rangle$
with $ \beta_0 =u\,.  $ Moreover, we have $df= \ejbh dz g$ with $ g =
2\pi j e^{\pi j ux}\,, $ and $ f = e^{j\frac{\beta}2}(-\frac j {2u} +
\pi iy) \frac{g}{\pi}\,.  $ With the same methods as before (with the
obvious adaptions to the situation of a cylinder), we consider for
$a\in \R$ all cylinder with $u\ge a$, and obtain again for
\[
B = \frac12( u + a i  - \sqrt{u^2-a^2})
\]
holomorphic sections with multiplier.
The corresponding frequencies are
\[
\delta_+ =\frac u2,  \quad \delta_- = \frac u2 + a i \in\Gamma^*_{0,B} =\{\delta\in u\Z + i \R +\frac{\beta_0}2 \mid |\delta-B| = \frac{|\beta_0|}2\}
\]
and the monochromatic holomorphic sections with multiplier $h^{0,B}$
are
\[
\alpha_\pm = \frac 1 u e^{j\pi ux} (u \pm j a - k \sqrt{u^2-a^2}) e^{\pi i(\sqrt{u^2-a^2} x \mp ay)}\,.
\] 
Again, we apply Theorem \ref{thm:Darboux transforms} with constants
$m_+ = m_-=1$, and obtain, after a similar computation as in the case
of rectangular tori, the monochromatic Darboux transforms of a
cylinder for $h^{(0,B)}$ as
\[
\hat f = \ejbh( j\tau_0 + i \tau_1) \frac{g}{\pi} = 2(-e^{2\pi jux} \tau_0 + k \tau_1)
\]
with real valued functions
\[
\tau_0(x,y) =\frac 1u( -\frac 12 + \frac 1{\hat R})
\]
and
\[
\tau_1(x,y) = \pi y + \frac 1{a\hat R}\left(\sin \tilde y(1-\frac{a^2}{u^2}) + \cos \tilde y  (\frac a{u^2}\sqrt{u^2-a^2})\right)\,.
\]
Here, we have with  $\tilde y = 2\pi n v y$
\[
\hat R(y) = \frac{u^2}{4a^2} R = 1 - (1-\frac{a^2}{u^2})\cos\tilde y + \frac a {u^2} \sqrt{u^2-a^2}\sin\tilde y\,.
\]
and thus, both $\tau_0$ and $\tau_1$ only depend on $y$.  In
particular, $\hat f$ is a surface of revolution in the 3--space
spanned by $1, j, k$, and obviously, $\hat f$ is a round cylinder
whenever $u=a$ for $a\in\R$.

\begin{figure}
\includegraphics[height=4.2cm]{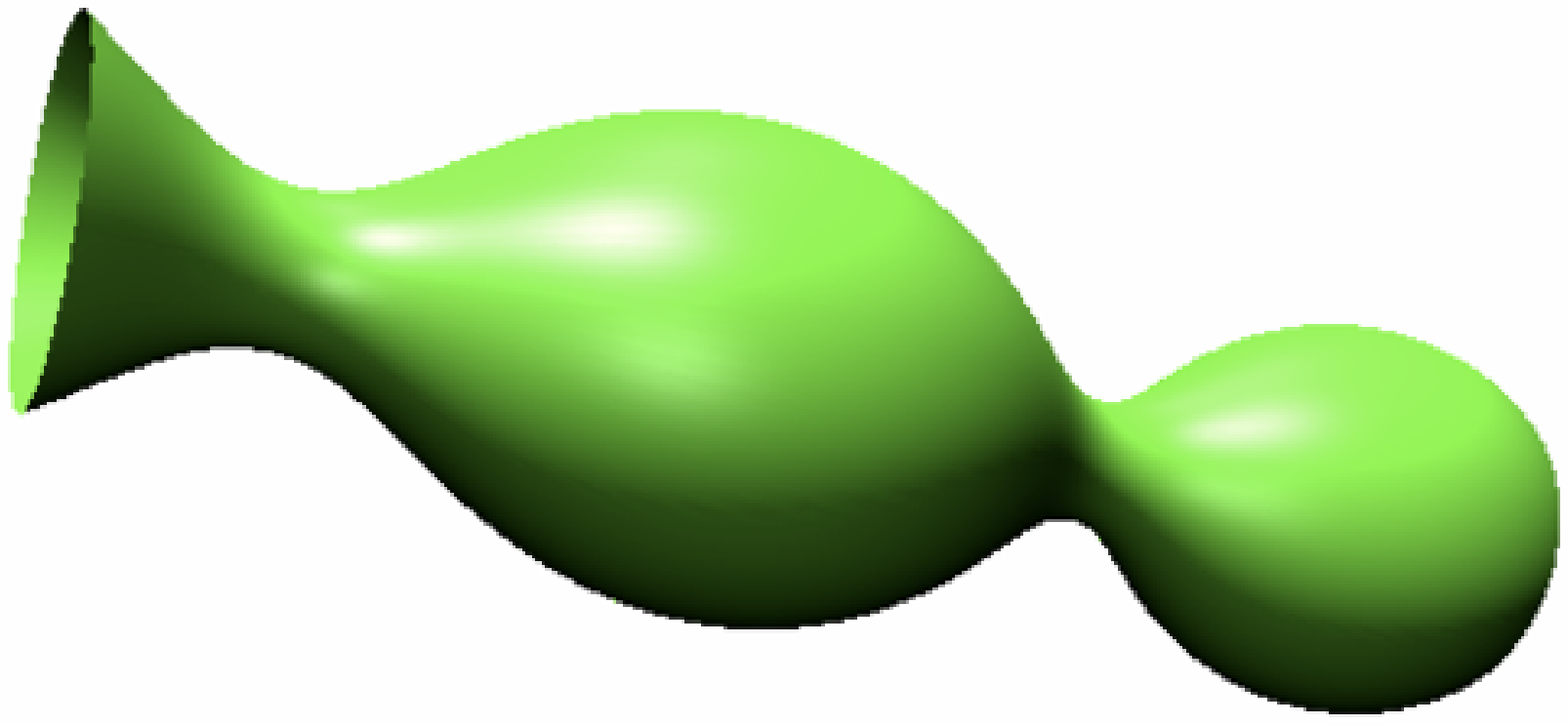}
\includegraphics[height=4.2cm]{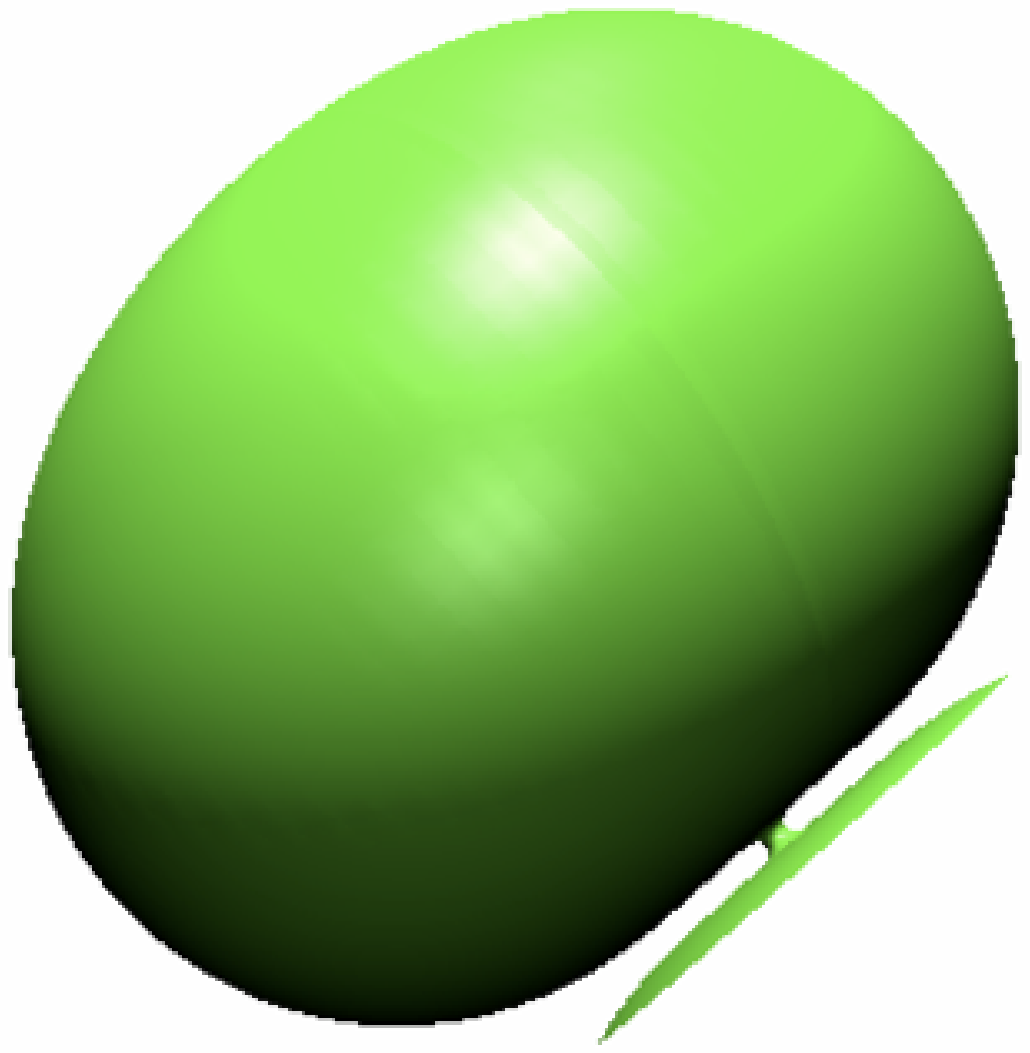}
\caption{polychromatic Darboux transforms of cylinder of radius 2.1 and 2.9}
\end{figure}

We now compute the mean curvature of $\hat f$. To that end, we observe
that $\hat f =f + T$ with
\begin{eqnarray*}
T % &=&  2e^{j\pi ux}\left( j \frac{1}{u\hat R} + i \frac 1{a\hat R}\left(\sin \tilde y(1-\frac{a^2}{u^2}) + \cos \tilde y  (\frac a{u^2}\sqrt{u^2-a^2})\right)
%\right)j e^{j\pi ux}\\
&=& \frac{2}{\hat R}\left(-\frac {e^{2\pi j ux}} u +  \frac ka \left(\sin \tilde y(1-\frac{a^2}{u^2}) + \cos \tilde y  (\frac a{u^2}\sqrt{u^2-a^2})\right)\right) 
\,, 
\end{eqnarray*}
and the left normal $\hat N$ of $\hat f$ is given (\ref{eq:hatN}) by
\[
\hat N = - \frac{ie^{-2\pi j ux} (\kappa_0^2-\kappa_1^2) + 2j\kappa_0\kappa_1}{\kappa_0^2 + \kappa_1^2}\,.
\]
with
\[
\kappa_0= -\frac 1 u \quad \text{ and } \quad
\kappa_1 =  \frac 1n \left(\sin \tilde y(1-\frac{a^2}{u^2}) + \cos \tilde y  (\frac a{u^2}\sqrt{u^2-a^2})\right) 
\]
real valued.  As before, we compute with (\ref{eq:hat H})
\newcommand{\ejx}{e^{2\pi j ux}} \newcommand{\ejmx}{e^{-2\pi j ux}}
\[
\hat H = i(\frac{\hat R\kappa_0}{\lambda} + \frac 1{4\tau_0}) 
\,,
\]
and, by evaluating at $y=0$ and $y=\frac{1}{2a}$, we see that $\hat H$
constant is equivalent to $a=u$. We summarize

\begin{theorem}
  For all $a\in\R, a>0$, the Darboux transformation gives a
  1--parameter family of cylinder of revolution which are not constant
  mean curvature cylinder in the 3--space.
\end{theorem} 

\begin{comment}
One easily verifies that for all $x\in [0,1]$ 
\[
\hat H(x,0) = -i\frac{ua^2}{2(2u^2-a^2)}
\]
and 
\[
\hat H(x,\frac 1{4a}) = i\frac{u^3 + ua\sqrt{u^2 - a^2} - 2ua^2}{2(u^2-a\sqrt{u^2-a^2})}
\]
which shows that $\hat H$ can only be constant if $u=a$, and thus
$\hat H = -i\frac u2$. This can easily be seen since $\hat H(x, 0) =
\hat H(x, \frac 1{4a})$ implies
\[
\sqrt{u^2-a^2} = \frac{a^2-u^2} a \,.
\]
But $2\le a\le u$ so that the only possible solution is $u=a$.
\end{comment}

%%% Local Variables: 
%%% mode: latex
%%% TeX-master: "doc"
%%% End: 
%\input{mudarboux}
%\input{spectralcurve}
%\input{parallel_section}
%\input{parallel}

%\input{flat_conn}

%\input{appendix}
\bibliographystyle{amsplain}
\bibliography{doc}
%erzeugen der bibliography durch bibtex ``documentname''

\end{document}